\documentclass[11pt]{article}
\usepackage{amsmath, amssymb, amsthm}
\usepackage{verbatim}
\usepackage{multicol}
\usepackage{enumerate}
\usepackage{comment}
\usepackage{dsfont}
\usepackage[none]{hyphenat}
\usepackage[unicode]{hyperref}
\hypersetup{
	colorlinks=true,
	linkcolor=blue,
	filecolor=magenta,      
	urlcolor=cyan,
	citecolor=blue
}
\usepackage{pgf}
\usepackage{tikz}
\usetikzlibrary{positioning,arrows,shapes,decorations.markings,decorations.pathreplacing,matrix,patterns}
\tikzstyle{vertex}=[circle,draw=black,fill=black,inner sep=0,minimum size=3pt,text=white,font=\footnotesize]
\usepackage{cleveref}

\date{}
\title{\vspace{-1.2cm} Note on the second eigenvalue of regular graphs}

\author{Igor Balla \thanks{Einstein Institute of Mathematics, Hebrew University of Jerusalem, Israel \emph{e-mail}: \textbf{iballa1990@gmail.com}. Research supported by SNSF Project 184522.}, Eero R\"aty\thanks{Ume\r{a} University, \emph{e-mail}: \textbf{\{eero.raty,istvan.tomon\}@umu.se}. ER is supported by a postdoctoral grant 213-0204 from the Olle Engkvist Foundation.}, Benny Sudakov\thanks{ETH Zurich, \emph{e-mail}: \textbf{benjamin.sudakov@math.ethz.ch}.
Research supported in part by SNSF grant 200021\_196965.}, 
	Istv\'an Tomon\footnotemark[2]
}

\theoremstyle{plain}
\newtheorem{theorem}{Theorem}[section]

\newtheorem{lemma}[theorem]{Lemma}

\Crefname{theorem}{Theorem}{Theorems}
\Crefname{definition}{Definition}{Definitions}
\Crefname{corollary}{Corollary}{Corollaries}
\Crefname{claim}{Claim}{Claims}
\Crefname{lemma}{Lemma}{Lemmas}
\Crefname{conjecture}{Conjecture}{Conjectures}
\Crefname{problem}{Problem}{Problems}
\Crefname{prop}{Proposition}{Propositions}

\theoremstyle{definition}

\begin{document}

\maketitle
\sloppy

\begin{abstract}
The goal of this expository note is to give a short, self-contained proof of nearly optimal lower bounds for the second largest eigenvalue of the adjacency matrix of regular graphs. This is a companion note to \cite{RST}.
\end{abstract}

\section{Introduction}
The Alon-Boppana theorem \cite{alon-boppana} is a fundamental result in spectral graph theory, which  states that if $G$ is an $n$-vertex  $d$-regular graph of diameter $D$, then the second largest eigenvalue $\lambda_2$ of the adjacency matrix is at least $2\sqrt{d-1}\cdot(1-o_D(1))$.  However, in case $D\in \{2,3\}$, which can happen for any $d\gg n^{1/3}$, the Alon-Boppana bound becomes meaningless. 

In a recent paper \cite{RST}, we addressed the growth rate of $\lambda_2$ for all values of $1\leq d\leq (1/2-\varepsilon)n$ by relating it to the so called \emph{positive discrepancy} or \emph{minimum bisection}, getting close to optimal bounds. The goal of this short note is to give a self-contained proof of these bounds. We believe that the behaviour of the smallest $\lambda_2$ among $d$-regular graphs is fairly surprising, so these results deserve their own exposition.

\begin{theorem}\label{thm:main}
Let $G$ be a $d$-regular $n$ vertex graph, and let $\lambda_2$ be the second largest eigenvalue of the adjacency matrix of $G$. Then for every $\varepsilon>0$, 
$$\lambda_2\geq \begin{cases}
\frac{\sqrt{d}}{2}(1-o(1)) &\mbox{ if } 1\leq d\leq n^{2/3}\\
\frac{n}{2d}-1 &\mbox{ if } n^{2/3}< d\leq n^{3/4}\\
(\varepsilon/4)^{1/3}\cdot d^{1/3}-1 &\mbox{ if } n^{3/4}< d\leq (1/2-\varepsilon)n,
\end{cases}$$
where the $o(1)$-term tends to 0 as $n\rightarrow \infty$. 
\end{theorem}
We did not put much effort into optimizing the constant terms. However, slightly more careful calculations yield that $\lambda_2\geq \sqrt{d/2}(1-o(1))$ if $d=o(n^{2/3})$, $\lambda_2\geq \frac{n}{d}(1-o(1))$ in case $d=o(n)$ and $n^{2/3}=o(d)$, and $\lambda_2\geq d^{1/3}\cdot (1-o(1))$ if $d=o(n)$.

In \cite{RST}, we discuss upper bounds as well (see also \cite{Balla21} for relations to constructions of \emph{equiangular lines}). Up to a constant factor, the random $d$-regular graph shows that the first inequality in Theorem \ref{thm:main} is tight. Our bounds are also tight if $d=\Theta(n^{3/4})$ or $d=\Theta(n)$. Finally, by relaxing the condition of being regular to almost regular (i.e.\ every degree is $d(1+o_n(1))$), the second inequality is also tight. Therefore, we conjecture that our Theorem \ref{thm:main} is tight up to constant factors for every $1\leq d\leq (1/2-\varepsilon)n$.
 Finally, the complete bipartite graph shows that $\lambda_2\geq 0$ cannot be improved if $d=n/2$.

We highlight that the first two inequalities are also implicit  in a work of Balla \cite{Balla21} (see Lemmas 3.13 and 3.14), while the third inequality was first obtained in \cite{Balla21}. This third inequality also follows from a recent result of Ihringer \cite{Ihringer} after a bit of work. We present essentially this latter argument. 

\section{Proof of the main theorem}

Let $G$ be a $d$-regular graph on $n$ vertices with adjacency matrix $A$, $1\leq d\leq (1/2-\varepsilon)n$, and let the eigenvalues of $A$ be $d=\lambda_1\geq \dots\geq \lambda_n$. Let $K$ be the largest index such that $\lambda_K\geq 0$. 
Since $G$ is not a complete graph its independence number is at least $2$, so by interlacing $\lambda_2\geq 0$ and $K\geq 2$. First, we prove the first two inequalities of Theorem \ref{thm:main} by considering the sum, quadratic sum, and cubic sum of eigenvalues.

The next equalities and inequalities follow by expressing $\mbox{tr}(A)=0$, $\mbox{tr}(A^2)=dn$, and $\mbox{tr}(A^3)\geq 0$ with the help  of the eigenvalues.

\begin{equation}\label{equ:sum1}
     d+n\lambda_2\geq d+\sum_{i=2}^{K} \lambda_i =\sum_{i=K+1}^{n} |\lambda_i|,
\end{equation}

\begin{equation}\label{equ:sum2}
 dn-d^2-n\lambda_2^2\leq dn-d^2-\sum_{i=2}^{K} \lambda_i^2 = \sum_{i=K+1}^{n} \lambda_i^2,
\end{equation}

\begin{equation}\label{equ:sum3}
    d^3+n\lambda_2^3\geq d^3+\sum_{i=2}^{K} \lambda_i^3 \geq  \sum_{i=K+1}^{n} |\lambda_i|^3.
\end{equation}
We may assume that $\lambda_2\leq \sqrt{d/2}$, otherwise we are done. This ensures that the left-hand-side of (\ref{equ:sum2}) is nonnegative. Let us use the simple inequality that for any $x_1,\dots,x_m\geq 0$, we have
\begin{equation}
    \left(\sum_{i=1}^{m} x_i\right)\cdot \left(\sum_{i=1}^{m}x_i^3\right)\geq \left(\sum_{i=1}^{m}x_i^2\right)^2.
\end{equation}
Then, we get from (\ref{equ:sum1},\ref{equ:sum2},\ref{equ:sum3}) that
$$(d+n\lambda_2)(d^3+n\lambda_2^3)\geq (dn-d^2-n\lambda_2^2)^2.$$
After expanding and simplifying, we arrive to the inequality
$$\lambda_2 ((d-\lambda_2)^2+2n\lambda_2)\geq d(n-2d).$$
Here, the left-hand-side is at most $\lambda_2 (d^2+2n\lambda_2)\leq 2\lambda_2\max\{d^2,2n\lambda_2\}$. By considering which term takes the maximum in the previous inequality, we get that 
$$\lambda_2\geq \min\left\{\frac{n}{2d}-1,\frac{\sqrt{d}}{2}\cdot \sqrt{1-2d/n}\right\}.$$
This proves the first two inequalities in Theorem \ref{thm:main}. 

\bigskip

Now let us turn to the third inequality, which we prepare with the following lemma. We may assume that $\lambda_2 \leq \sqrt{d}/2$, otherwise we are done.

	\begin{lemma}\label{lemma:product}
		$(1+\lambda_2)\cdot |\lambda_n|\geq \frac{d}{4}.$
	\end{lemma}

	\begin{proof}		
		Using that $\lambda_2 \leq \sqrt{d}/2$, we have $\sum_{i=2}^{K}\lambda_i^2\leq \frac{nd}{4}$, so 
		$$\frac{nd}{4}\leq \sum_{i=K+1}^{n}\lambda_i^2\leq |\lambda_n|\sum_{i=K+1}^{n}|\lambda_i|=|\lambda_n|\left(d+\sum_{i=2}^{K}\lambda_i\right)\leq |\lambda_n|(d+n\lambda_2) \leq |\lambda_n|n(1+\lambda_2).$$
		Here, the first inequality holds by (\ref{equ:sum2}) (since $d \leq n/2$), and the first equality by (\ref{equ:sum1}).
		
	\end{proof}

Set $p = \frac{d}{n}$, and define the matrix $M$ as 
\[ M = (1 - \alpha) I - \beta A  + \alpha J,\]
where $\alpha = \frac{p}{\lambda_{2} + p}$, $\beta = \frac{1}{\lambda_{2} + p}$, $I$ is the identity matrix, and $J$ is the all-1 matrix. Then the eigenvalues $\mu_{1}, \dots, \mu_{n}$ of $M$ are
\[\mu_{1} = 1 - \alpha - d\beta + n\alpha = 1-\alpha \geq 0 \] 
and
\[\mu_{i} = 1 - \alpha - \beta \lambda_{i} = \frac{\lambda_{2} - \lambda_{i}}{\lambda_{2} + p} \geq 0\]
for $i=2,\dots,n$. Hence $M$ is positive semi-definite, and thus the Hadamard product (also known as entrywise product) $M \circ M$ is also positive semi-definite by the Schur product theorem. By simple calculations, the Hadamard product $M \circ M$ is equal to
\[M \circ M = (1 - \alpha^2)I + (\beta^2 - 2\alpha \beta)A  + \alpha^2 J, \]
and hence its eigenvalues $\eta_{1}, \dots, \eta_{n}$ are
\[\eta_{1} = (1 - \alpha^2) + d\beta(\beta - 2\alpha) + n\alpha^2 = 1 - \alpha^2 + n\alpha(\beta - \alpha) = 1+ \frac{p(1-p)n - p^2}{(\lambda_2 + p)^2}\]
and 
\[\eta_{i} = (1- \alpha^2) + (\beta^2 - 2\alpha \beta)\lambda_{i} = 1 + \frac{(1-2p)\lambda_{i} - p^2}{(\lambda_2 + p)^2}\] 
for $i= 2,\dots, n$. Using that $M\circ M$ is positive semi-definite, we have $\eta_i\geq 0$ for $i\in [n]$. Furthermore, as $d < \left( \frac{1}{2} - \varepsilon \right) n$, the inequality $1-2p>2\varepsilon>0$ also holds. Therefore, $\eta_{n} \geq 0$  implies that $(\lambda_2 + p)^2 \geq (1-2p) \vert \lambda_{n} \vert + p^2 \geq 2\varepsilon \vert \lambda_{n} \vert $. Multiplying both sides of the previous inequality by $(1+\lambda_2)$ and using the fact that $\lambda_2+p\leq \lambda_2+1$, we obtain 
$$(1+\lambda_2)^3\geq2\varepsilon |\lambda_n|(1+\lambda_2)\geq  \frac{\varepsilon d}{2}.$$
Here the last inequality follows by Lemma \ref{lemma:product}. This gives the third inequality in Theorem \ref{thm:main}.

\end{document}